\documentclass{amsart}
\usepackage[T1]{fontenc}
\usepackage[utf8]{inputenc}

\usepackage{hyperref}

\usepackage[english]{babel}
\usepackage[utf8]{inputenc}

\theoremstyle{definition}
\newtheorem{definition}{Definition}[section]

\newtheorem{remark}[definition]{Remark}

\theoremstyle{plain}
\newtheorem{theorem}[definition]{Theorem}
\newtheorem{lemma}[definition]{Lemma}

\newtheorem{corollary}[definition]{Corollary}

\DeclareMathOperator{\rk}{rk}

\DeclareMathOperator{\NS}{NS}
\DeclareMathOperator{\Aut}{Aut}
\DeclareMathOperator{\cha}{char}

\DeclareMathOperator{\Nef}{Nef}
\DeclareMathOperator{\sdeg}{sdeg}

\newcommand{\QQ}{\mathbb{Q}}
\newcommand{\FF}{\mathbb{F}}
\newcommand{\RR}{\mathbb{R}}
\newcommand{\CC}{\mathbb{C}}
\newcommand{\ZZ}{\mathbb{Z}}
\newcommand{\NN}{\mathbb{N}}

\begin{document}

\title[Automorphisms of Salem degree $22$]{Automorphisms of Salem degree $22$ on supersingular K3 surfaces of\\ higher Artin invariant}

\author[Simon Brandhorst]{Simon Brandhorst}

\begin{abstract}
 We give a short proof that every supersingular K3 surface (except possibly in characteristic $2$ with Artin invariant $\sigma=10$) has an automorphism of Salem degree 22. 
 In particular an infinite subgroup of the automorphism group does not lift to characteristic zero. 
 The proof relies on the case $\sigma=1$ and the cone conjecture for K3 surfaces.
\end{abstract}

\maketitle

\section{Introduction}
A Salem number is a real algebraic integer $\lambda>1$ which is conjugate to $1/\lambda$ and all whose other conjugates lie on the unit circle. Its minimal polynomial is called a Salem polynomial.
Salem numbers arise naturally in algebraic geometry as follows: If $X$ is a projective surface over an algebraically closed field $k$ and $f\!\!:X\rightarrow X$ an automorphism, then the characteristic polynomial
\[\chi(f^*|H^2_{\acute{e}t}(X,\QQ_\ell(1))) \quad (\ell\neq \cha k)\]
factors as a product of cyclotomic polynomials and at most one Salem polynomial $s(x)$ \cite{esnault-srinivas:algebraic_entropy}. We call the degree of the Salem factor $s(x)$ the \emph{Salem degree} of $f$. 
Let $H$ be an ample polarization of $X$. Since the order of $f^*$ is finite on 
\[\langle f^{*k}(H) \mid k \in \ZZ \rangle^\perp \subseteq H^2_{\acute{e}t}(X,\QQ_\ell(1))\]
by \cite{esnault-srinivas:algebraic_entropy}, we get that $\ker s(f^*|H^2_{\acute{e}t}(X,\QQ_\ell(1)))$ is contained in the ($\ell$-adic) Néron-Severi group $\NS(X)\otimes \QQ_\ell$ of $X$. In particular, we can bound the Salem degree of an automorphism by the Picard number $\rho(X)$. For a K3 surface $X$ it is at most $\rho(X)\leq h^{1,1}(X)=20$ in characteristic $0$ by Lefschetz' Theorem on (1,1)-classes. However, in positive characteristic supersingular K3 surfaces have $\rho(X)=22$. Indeed:
\begin{theorem}\label{thm:sigma=1}\cite{cantat:dynamical_degrees,esnault-oguiso:non-liftability,esnault-oguiso-yu:salem22,shimada:salem22,schuett:salem22,brandhorst:salem22}
The supersingular K3 surface $X/k$, $k=\overline{k}$, $\cha k>0$, of Artin invariant one has an automorphism of Salem degree 22. 
\end{theorem}
Note that the characteristic polynomial of $f^*$ is stable under (good) specialization by standard comparison theorems. 
This observation leads to the interesting feature that an automorphism of Salem degree 22 is not geometrically liftable to characteristic zero (see \cite{esnault-oguiso:non-liftability} for details).
This is in sharp contrast to the case of non-supersingular K3 surfaces in odd characteristic. There one can allways lift a finite index subgroup of the automorphism group to characteristic zero (cf \cite[Thm. 3.2]{jang:lifting_of_automorphisms}).\\

Supersingular K3 surfaces are classified by their Artin invariant $1 \leq \sigma \leq 10$. For fixed Artin invariant $\sigma$ they form a family of dimension $\sigma-1$, while the supersingular K3 surface of Artin invariant $\sigma=1$ is unique (cf. \cite{rudakov_shafarevic:k3_finite_fields,ogus:83}). 
The main purpose of this note is to extend Theorem \ref{thm:sigma=1} to \emph{all} supersingular K3 surfaces. 
\begin{theorem}[Main Theorem]\label{thm:salem22general}
 Let $Y/k$ be a supersingular K3 surface over an algebraically closed field such that the crystalline Torelli theorem holds for $Y$. Then $Y$ has an automorphism of Salem degree 22.
\end{theorem}
\begin{remark}
Set $p=\cha k$ and $\sigma=\sigma(Y)$. The crystalline Torelli is proven for $p>3$ in \cite[Thm. I]{ogus:83} and for $p=2$ and $\sigma<10$ and for $p=3$ and $\sigma<6$ (at the end of \cite{rudakov_shafarevic:k3_finite_fields}).
For $p=3$ the main theorem is proved in \cite{shimada:salem22}. Hence the only open case left is $p=2$ and $\sigma=10$.
The main step in the proof is a reduction to Theorem \ref{thm:sigma=1}.
\end{remark}
\noindent
In a recent preprint \cite{Yu:salem22} Yu gives an independent proof of the main theorem for $p>3$ using genus one fibrations.
However, I believe the new proof to be of independent interest, as it is shorter and characteristic free. In particular the result for $p=2, \sigma>1$ is new.

\textbf{Acknowledgments.} I thank Hélène Esnault, Víctor González-Alonso, Keiji Oguiso, Matthias Schütt, and Xun Yu for discussions and comments on this work.

\section{Preliminaries}
A lattice $L$ is a finitely generated free abelian group equipped with a nondegenerate, integer valued bilinear form.
It is called even if $x^2 \in 2\ZZ$ for all $x\in L$.
The dual lattice is $L^\vee=\{x \in L \otimes \QQ: x.L\subseteq \ZZ\}$ and the discriminant group $L^\vee\!/L$ of an even lattice $L$ is equipped with the quadratic form 
\[q:L^\vee\!/L \rightarrow \QQ/2\ZZ,\quad  \overline{x} \mapsto x^2 \mod 2\ZZ.\]
A supersingular K3 lattice $N$ is an even lattice of signature $(1,21)$ and discriminant group $N^\vee\!/N\cong \FF_p^{2\sigma}$. If $p=2$, we require furthermore that it is of type I, i.e. $x^2 \in \mathbb{Z}$ for $x \in N^\vee$. Such a lattice is determined up to isometry by $p$ and $\sigma$ (cf. \cite[sect. 1]{rudakov_shafarevic:k3_finite_fields}). 
Let $X$ be a K3 surface defined over an algebraically closed field $ k $ of characteristic $p$. Recall that $X$ is said to be {\em supersingular} if
\[\rho\left(X\right) = \rk \NS\left(X\right) = 22.\]
Then the N\'eron-Severi lattice $\NS\left(X\right)$ is a supersingular K3 lattice for $p=\cha k$ and $1 \leq \sigma \leq 10$ (cf. \cite[sect. 8]{rudakov_shafarevic:k3_finite_fields}). We call $\sigma$ the {\em Artin invariant} of $X$. 

For the readers' convenience we give a proof of the following well known
\begin{lemma}
 There is an embedding 
 $N_{p,\sigma} \hookrightarrow N_{p',\sigma'}$
 of supersingular K3 lattices if and only if $p=p'$ and $\sigma'\leq \sigma$. 
\end{lemma}
\begin{proof}
 The only if part follows from the fact that if $A \subset B$ are two lattices of the same rank, then 
 \[\det A= [B:A]^2 \det B.\]
 In this situation 
 \[A \hookrightarrow B \hookrightarrow B^\vee \hookrightarrow A^\vee\]
 and $B/A$ is a totally isotropic subspace of $A^\vee \!/A$. Now, if $A$ is $2$-elementary of type $I$, then, since $B^\vee \subseteq A^\vee$, so is $B$. 
 Let $p\neq 2$. Then the quadratic space $N_{p,10}^\vee/N_{p,10}\cong \FF_p^{20}$ contains an isotropic line since it is of dimension greater two. As above this line corresponds to an overlattice $N$ of $N_{p,10}$ which is hyperbolic and $|N^\vee\!\!/N|=p^{18}$. Since subquotients of vector spaces are vector spaces, we see that $N^\vee\!\!/N\cong \FF_p^{18}$. Then $N\cong N_{p,9}$ is in fact a supersingular K3 lattice. 
 Continuing in the same way, we get a chain of overlattices
 \[N_{p,10} \subseteq N_{p,9} \subseteq\cdots \subseteq N_{p,1}.\]
 Note that the process stops at $\sigma=1$ since there is no isotropic line in the discriminant group. This is in accordance with the fact that there is no even unimodular lattice of signature $(1,21)$. 
 For $p=2$ the discriminant form is isomorphic to a direct sum of forms of type $q(x,y)=x^2+xy+y^2 \mod 2 \ZZ$ and the existence of an isotropic vector follows as long as there are at least two summands, i.e., $\sigma>1$. 
 Since everything is contained in $N_{p,10}^\vee$, the constructed lattices stay of type I. 
\end{proof} 

Let $L$ be an even lattice of signature $(1,n)$ and denote by
$O^+(L)$ the subgroup of isometries preserving the two connected components of the positive cone.
Set 
$$V_L = \left\{x \in L\otimes\mathbb{R} \,|\, x^2 > 0 \, \text{ and } \, \forall r \in L \text{ with } r^2=-2\colon \left(r,x\right) \neq 0\right\}.$$
According to \cite[Proposition 1.10]{ogus:83}, the set $V_L$ is open and each of its connected components meets $L \subset L\otimes\mathbb{R}$. These connected components of $V_L$ are called {\em chambers} of $V_L$. Each point $r$ of length $-2$ induces an orthogonal reflection 
\[\delta_r: L \rightarrow L \quad x\mapsto x+\langle x.r\rangle r\] 
along the hyperplane $r^\perp$.
The Weyl group $W(L)\subseteq O(L)$ is the group generated by all orthogonal reflections along a $(-2)$-hyperplane.
It acts transitively on the set of chambers. \\

If $L = \NS(X)$ for a K3 surface $X$, then one of the chambers is the ample cone. Its closure is the nef cone $\Nef(X)$. Classes of smooth rational curves are called nodal. By adjunction they are of square $(-2)$ and they are exactly the rays of the effective cone. Note that if $r^2=-2$, then by Riemann-Roch either $r$ or $-r$ is effective but they are not necessarily nodal.

\begin{theorem}[Cone conjecture]\cite[Thm. 6.1]{lieblich_maulik:cone_theorem}\label{thm:cone}
 Let $X$ be a K3 surface over an algebraically closed field $k$. If $X$ is supersingular suppose that crystalline Torelli holds for $X$. 
 Let $\Gamma(X) \subseteq O^+(\NS(X))$ be the subgroup preserving the nef cone.
 Then $\Gamma(X)\cong O^+(\NS(X))/W(\NS(X))$ and
 \begin{enumerate}
  \item The natural map $\Aut(X) \rightarrow \Gamma(X)$ has finite kernel and cokernel. 
  \item The group $\Aut(X)$ is finitely generated.
  \item The action of $\Aut(X)$ on $\Nef(X)$ has a rational polyhedral fundamental domain.
  \item The set of orbits of $\Aut(X)$ in the nodal classes of $X$ is finite.
 \end{enumerate}
\end{theorem}
Over $\CC$ the theorem follows from the strong Torelli theorem by work of Sterk \cite[Thm. 01.]{sterk:finiteness}. Then, for K3 surfaces of finite height in arbitrary characteristic one can lift $X,\NS(X)$ and a finite index subgroup of $\Aut(X)$ to characteristic zero and apply the cone theorem there. For supersingular K3 surfaces one has to use the crystalline Torelli Theorem. 
In this case $\Aut(X) \rightarrow \Gamma$ is injective and its image contains the finite index subgroup $\ker ( \Gamma \rightarrow O(\NS^\vee \! \! /\NS))$. 

\begin{lemma}\cite[p. 169]{salem:powers}
 If $\lambda$ is a Salem number of degree $d$ then $\lambda^n$, $n\in \NN$ is a Salem number of the same degree.
\end{lemma}
\begin{proof}
 Denote the Galois conjugates of $\lambda=\lambda_1$ by $\lambda_i$ $i=1,\dots,n$. Then the Galois conjugates of $\lambda_1^n$ are the $\lambda_i^n$. In particular $\lambda_1^n$ is a Salem number. It remains to check that its conjugates are all distinct. Suppose that $\lambda_i^n=\lambda_k^n$. After applying a Galois conjugation we may assume that $i=1$. In particular, $1<\lambda_1^n=\lambda_k^n$. Now, $|\lambda_k|>1$ is the unique conjugate of absolute value greater one, i.e. $k=1$.
\end{proof}

\begin{corollary}\label{coro:salem_ns}
 The maximum occurring Salem degree of an automorphism of a K3 surface $X$ over an algebraically closed field depends only on the isometry class of $\NS(X)$, given that the cone conjecture holds for $X$. 
\end{corollary}
\begin{proof}
 Since any power of a Salem number of degree $d$ remains a Salem number of this degree, we may pass to a finite index subgroup.
 Combining this with part (1) of the cone conjecture, we get that the maximum occurring Salem degree of an automorphism of $X$ depends only on $\Gamma(X)$. Now, $\Gamma(X)$ depends up to conjugation by an element of the Weyl group only on the isometry class of $\NS(X)$. In particular, the maximal Salem degree of an automorphism of $X$ depends only on $\NS(X)$.
\end{proof}

\section{Proof of the main theorem}
 \begin{lemma}\label{lem:supersingular_embedding}
 Let $N\subseteq L$ be two lattices of the same rank and $G\subseteq O(L)$ a subgroup.
 Then
 \[[G:O(N) \cap G]<\infty\]
 where we view $O(N)$ and $O(L)$ as subgroups of $O(N \otimes \RR)$. 
\end{lemma}
\begin{proof}
 Since the ranks coincide, the index $n=[L:N]$ is finite and
 \[nL \subseteq N \subseteq L.\]
 Any isometry of $L$ preserves $nL$ hence we get a map
 \[\varphi \colon G \rightarrow \Aut(L/nL).\]
 Set $K=\ker \varphi$, which is a finite index subgroup of $G$. 
 To see that $K\subseteq O(N)$ as well, recall that an isometry $f$ of $O(nL)$ extends to $O(N)$ iff $f(N/nL)=N/nL$. Indeed, $f|L/nL=id|L/nL$ for $f \in K$, by definition.
\end{proof}
The following is a generalization of \cite[Thm. 1.2]{Yu:salem22} where the existence of at least one elliptic fibration on $X$ with infinite automorphism group is assumed. We can drop this condition. 
\begin{theorem}\label{thm:Saleminequality}
Let $X/k,Y/k'$ be two K3 surfaces over algebraically closed fields $k,k'$ satisfying the cone conjecture. 
Suppose that $\rho(X)=\rho(Y)$ and that there is an isometric embedding
\[\iota:\NS(Y)\hookrightarrow \NS(X).\]
Then $\sdeg(X)\leq \sdeg(Y)$ where 
$$\sdeg (X)=\max\{\mbox{Salem degree of }f \mid f \in \Aut(X)\}.$$
\end{theorem}
\begin{proof}
 Denote by $\Nef(X)$ and $\Nef(Y)$ the nef cones of $X$ and $Y$.
 Any chamber of the positive cone of $\NS(X)$ is contained in the image of a unique chamber of the positive cone of $\NS(Y)$. Since the Weyl group acts transitively on the chambers, we can find an element $\delta\in W(\NS(X))$ of the Weyl group such that $\Nef(X)\subset \iota'_\RR(\Nef(Y))$ where $\iota'=\delta\circ \iota$. 
 To ease notation we identify $\NS(Y)$ with its image under $\iota'$.
 By the preceding Lemma $[\Gamma(X):\Gamma(X)\cap O(\NS(Y))]$ is finite, and since $\Nef(X) \subseteq \Nef(Y)$, we get that $\Gamma(X) \cap O(\NS(Y)) \subseteq \Gamma(Y)$. Now, by the cone Theorem \ref{thm:cone} and the proof of Corollary \ref{coro:salem_ns}
\begin{align*}
\sdeg(X) = \sdeg(\Gamma(X)) &= \sdeg(\Gamma(X)\cap O(NS(Y)))\\
&\leq \sdeg(\Gamma(Y))= \sdeg(Y).
\end{align*}
\end{proof}

 \begin{proof}[Proof of Theorem. \ref{thm:salem22general}]
  If $X/k$ and $Y/k$ are supersingular K3 surfaces with $\sigma(X)\leq \sigma(Y)$, then $\NS(Y) \hookrightarrow \NS(X)$ by Lemma \ref{lem:supersingular_embedding}. Combining the $\sigma=1$ case (Thm. \ref{thm:sigma=1}) and the previous theorem we get that $22=\sdeg(X)\leq \sdeg (Y)\leq 22$.
 \end{proof}
 
 The converse inequality in Theorem \ref{thm:Saleminequality} is false in general. See \cite[rmk. 7.3]{Yu:salem22} for examples.

\bibliography{salem22}

\begin{thebibliography}{10}

\bibitem{cantat:dynamical_degrees}
J.~Blanc and S.~Cantat, \emph{Dynamical degrees of birational transformations
  of projective surfaces}, J. Amer. Math. Soc. \textbf{29} (2016), no.~2,
  415--471.

\bibitem{brandhorst:salem22}
S.~Brandhorst, \emph{Dynamics on supersingular {K}3 surfaces and automorphisms
  of {Sa}lem degree 22}, Nagoya Math. J.  (2016), 1--15.

\bibitem{esnault-oguiso:non-liftability}
H.~Esnault and K.~Oguiso, \emph{Non-liftability of automorphism groups of a
  {K}3 surface in positive characteristic}, Math. Ann. \textbf{363} (2015), 
no.~3-4,  1187--1206.

\bibitem{esnault-oguiso-yu:salem22}
H.~Esnault, K.~Oguiso, and X.~Yu, \emph{Automorphisms of elliptic {K}3 surfaces
  and {S}alem numbers of maximal degree}, Alg. Geom. \textbf{3} (2016), no.~4,
  496--507.

\bibitem{esnault-srinivas:algebraic_entropy}
H.~Esnault and V.~Srinivas, \emph{Algebraic versus topological entropy for
  surfaces over finite fields}, Osaka J. Math. \textbf{50} (2013), no.~3,
  827--846.

\bibitem{jang:lifting_of_automorphisms}
J.~Jang, \emph{A Lifting of an Automorphism of a {K}3 Surface over Odd
  Characteristic}, Int. Math. Res. Notices  (2016).

\bibitem{lieblich_maulik:cone_theorem}
M.~Lieblich and D.~Maulik, \emph{A note on the cone conjecture for K3 surfaces
  in positive characteristic},  (2011).

\bibitem{ogus:83}
A.~Ogus, \emph{A crystalline {T}orelli theorem for supersingular {$K3$}
  surfaces}, in: Arithmetic and geometry, {V}ol. {II}, Vol.~36 of {Progr.
  Math.}, 361--394, Birkh\"auser Boston, Boston, MA (1983).

\bibitem{rudakov_shafarevic:k3_finite_fields}
A.~N. Rudakov and I.~R. Shafarevich, \emph{Surfaces of type {$K3$} over fields
  of finite characteristic}, in: Current problems in mathematics, {V}ol.~18,
  115--207, Akad. Nauk SSSR, Vsesoyuz. Inst. Nauchn. i Tekhn. Informatsii,
  Moscow (1981).

\bibitem{salem:powers}
R.~Salem, \emph{Power series with integral coefficients}, Duke Math. J.
  \textbf{12} (1945), 153--172.

\bibitem{schuett:salem22}
M.~Sch\"utt, \emph{Dynamics on supersingular {K}3 surfaces}, Comment. Math.
  Helv.  (2016), 705--719.

\bibitem{shimada:salem22}
I.~Shimada, \emph{Automorphisms of supersingular {$K3$} surfaces and {S}alem
  polynomials}, Exp. Math. \textbf{25} (2016), no.~4,  389--398.

\bibitem{sterk:finiteness}
H.~Sterk, \emph{Finiteness results for algebraic {$K3$} surfaces}, Math. Z.
  \textbf{189} (1985), no.~4,  507--513.

\bibitem{Yu:salem22}
X.~Yu, \emph{{E}lliptic fibrations on {K}3 surfaces and {S}alem numbers of
  maximal degree},  (2016).

\end{thebibliography}
\address{Insitut für Algebraische Geometrie, Leibniz Universität Hannover,
	Welfengarten 1, 30167 Hannover, Germany}\\
\email{brandhorst@math.uni-hannover.de}

\end{document}